\documentclass[reqno, 10pt]{amsart}

\usepackage{amsmath,amssymb,amsthm}

\usepackage{tikz}

\usepackage{caption}
\usepackage{graphicx}
\usepackage{graphics}
\usepackage{tikz}	
\usetikzlibrary{cd,positioning, shapes}	
\usepackage{comment}
\usepackage{quiver}
\usepackage{mdframed}
\usepackage{empheq}
\usepackage[hidelinks]{hyperref}

\newtheorem{theorem}{Theorem}[section]
\newtheorem{proposition}[theorem]{Proposition}
\newtheorem{question}[theorem]{Question}
\newtheorem{corollary}[theorem]{Corollary}
\newtheorem{conjecture}[theorem]{Conjecture}
\newtheorem{lemma}[theorem]{Lemma}

\theoremstyle{definition}

\newtheorem{example}[theorem]{Example}

\newtheorem{remark}[theorem]{Remark}

\newtheorem{claim}{Claim}

\begin{document}

\title[]{Monotone paths on acyclic 3-regular graphs}

\author[]{Fran\c{c}ois Cl\'ement}
\author[]{Dan Guyer}
\address{Department of Mathematics, University of Washington, Seattle}
\email{fclement@uw.edu} 
\email{dguyer@uw.edu}

\begin{abstract}
Motivated by trying to understand the behavior of the simplex method, Athanasiadis, De Loera and Zhang~\cite{Loera2022} provided upper and lower bounds on the number of the monotone paths on 3-polytopes. For simple 3-polytopes with $2n$ vertices, they showed that the number of monotone paths is bounded above by $(1+\varphi)^n$, with $\varphi$ being the golden ratio. We improve the result and show that for a larger family of graphs the number is bounded above by $c \cdot 1.6779^n$ for some universal constant $c$. Meanwhile, the best known construction and conjectured extremizer is approximately $\varphi^n$.
\end{abstract}

\maketitle
\section{Introduction and Results}
\subsection{Introduction}
For a generic linear function $f:\mathbb{R}^d \rightarrow \mathbb{R}$ and a polytope $P \subseteq \mathbb{R}^d$, an $f$-monotone path is a directed path on the graph of $P$ that is strictly increasing with respect to $f$. Monotone paths have appeared in a variety of problems, such as the Hirsch conjecture bounding the diameter of a polytope~\cite{Ziegler}, fractional power series solutions of systems of polynomial equations~\cite{MCDONALD1995213}, and Ramsey numbers~\cite{Sudakov}.
The number of monotone paths and coherent monotone paths were studied for geometric triangulations in the plane~\cite{Dumitrescu2016}, and families of polytopes such as cross polytopes~\cite{Black2023}, hypersimplices~\cite{poullot2024} and general polytopes~\cite{Loera2022,juhnke2025}.

We are motivated by the work of Athanasiadis--de Loera--Zhang~\cite{Loera2022}. They asked for the largest possible number of maximal monotone paths on a simple $3$-polytope with $2n$ vertices and proved an upper bound of $(1+\varphi)^n$. By Steinitz's Theorem~\cite{Steinitz}, the family of graphs of these polytopes corresponds exactly to $3$-regular, 3-connected, planar, directed graphs.  
We will study the maximal number of source-to-sink paths on $3$-regular, $3$-edge connected, acyclic graphs; our arguments do not require the planarity condition. Below, we include a drawing of the graph of the truncated tetrahedron as a 3-regular, 3-connected, acyclic graph.
    \[\begin{tikzcd}[column sep=0.24in]
	\bullet & \bullet & \bullet & \bullet & \bullet & \bullet & \bullet & \bullet & \bullet & \bullet & \bullet & \bullet
	\arrow[from=1-1, to=1-2]
	\arrow[curve={height=-6pt}, from=1-1, to=1-3]
	\arrow[shift right, curve={height=24pt}, from=1-1, to=1-12]
	\arrow[from=1-2, to=1-3]
	\arrow[shift right, curve={height=6pt}, from=1-2, to=1-8]
	\arrow[from=1-3, to=1-4]
	\arrow[from=1-4, to=1-5]
	\arrow[curve={height=6pt}, from=1-4, to=1-6]
	\arrow[from=1-5, to=1-6]
	\arrow[curve={height=-12pt}, from=1-5, to=1-11]
	\arrow[from=1-6, to=1-7]
	\arrow[from=1-7, to=1-8]
	\arrow[curve={height=-6pt}, from=1-7, to=1-9]
	\arrow[from=1-8, to=1-9]
	\arrow[from=1-9, to=1-10]
	\arrow[from=1-10, to=1-11]
	\arrow[curve={height=6pt}, from=1-10, to=1-12]
	\arrow[from=1-11, to=1-12]
\end{tikzcd}\]
Observe that this graph has 21 source-to-sink paths.

\subsection{Main results}

The main result of this paper is the following theorem which improves the upper bound of $2F_{2n-1} \approx (1+\varphi)^n$ source-to-sink paths on a simple 3-polytope from Proposition~4.7 of  Athanasiadis--de Loera--Zhang~\cite{Loera2022}.

\begin{theorem}\label{th:opt}
    The maximal number of source-to-sink paths on a 3-regular, 3-edge connected acyclic graph on $2n$ vertices is upper bounded by $c \cdot 1.6779^n$, where $c$ is a universal constant.
\end{theorem}

One result that will be very useful in our proof of Theorem~\ref{th:opt} is that at least one of the graphs that maximize the number of source-to-sink paths has to admit a Hamiltonian path. This Hamiltonian path will be the `backbone' of the argument.

\begin{theorem}\label{th:Ham}
    Among all 3-regular, $3$-edge connected acyclic graphs, the number of source-to-sink paths is maximized by those with a directed Hamiltonian path.
\end{theorem} 

This result enables a characterization of 3-edge connectivity for graphs on a Hamiltonian path, described in Proposition~\ref{3conn}, which we use often in the proof arguments. This result also serves as a potential starting point for future progress.

In light of Theorem~\ref{th:opt}, we suspect that Conjecture~4.6 of ~\cite{Loera2022} holds regardless of the planarity of the underlying graph. Therefore, we make the following conjecture which uses the Fibonacci sequence defined by $F_1=F_2=1$ and $F_{n}=F_{n-1}+F_{n-2}$ for $n\geq 3$ to bound the number of monotone paths on a 3-regular, 3-edge connected graph.

\begin{conjecture}\label{conj:3conn3reg}
    Any 3-regular, 3-edge connected, acyclic directed graph with a unique source and a unique sink on $2n$ vertices  has at most $F_{n+2}+1$ source-to-sink paths where $F_{n+2}$ denotes the $(n+2)^{\mathrm{th}}$ Fibonacci number.
\end{conjecture}

We emphasize once more that this is simply Conjecture~4.6 of ~\cite{Loera2022} without the restriction on the graph having to be planar. The conjectured optimum is the graph of the wedge of an $(n+1)$-gon over an edge (see Figure 7 in ~\cite{Loera2022}). Heuristically speaking, shorter edges induce more source-to-sink paths, while longer edges help maintain connectivity along the Hamiltonian path. While one can draw examples of graphs with a mix of shorter and longer edges, as in the graph of the truncated tetrahedron, this conjecture posits that maintaining a more uniform edge length maximizes the source-to-sink path count.
We sketch this conjectured optimum on twelve vertices below.
   \[
  \begin{tikzcd}[column sep=0.24in]
	\bullet & \bullet & \bullet & \bullet & \bullet & \bullet & \bullet & \bullet & \bullet & \bullet & \bullet & \bullet
	\arrow[from=1-1, to=1-2]
	\arrow[curve={height=-12pt}, from=1-1, to=1-3]
	\arrow[curve={height=12pt}, from=1-1, to=1-12]
	\arrow[from=1-2, to=1-3]
	\arrow[shift left, curve={height=-12pt}, from=1-2, to=1-5]
	\arrow[from=1-3, to=1-4]
	\arrow[from=1-4, to=1-5]
	\arrow[shift left, curve={height=-12pt}, from=1-4, to=1-7]
	\arrow[from=1-5, to=1-6]
	\arrow[from=1-6, to=1-7]
	\arrow[shift left, curve={height=-12pt}, from=1-6, to=1-9]
	\arrow[from=1-7, to=1-8]
	\arrow[from=1-8, to=1-9]
	\arrow[shift left, curve={height=-12pt}, from=1-8, to=1-11]
	\arrow[from=1-9, to=1-10]
	\arrow[from=1-10, to=1-11]
	\arrow[curve={height=-12pt}, from=1-10, to=1-12]
	\arrow[from=1-11, to=1-12]
\end{tikzcd}
\]
Indeed, this example has $F_8+1=22$ source-to-sink paths.

\subsection{The big picture}
Theorem~\ref{th:opt} and the intricate structure of the conjectured extremizer naturally suggest investigating a more general problem. For a 3-regular acyclic directed graph on $2n$ vertices that is $\ell$-edge connected, how many maximal paths are possible?  Some of our results are summarized in the following table. The results, conjectures, optimizers and conjectured optimizers for connected and 2-edge connected graphs are in Subsection~\ref{subsec:supp}, while the proof of Theorem~\ref{th:opt} is in Section~\ref{sec:opti}. We remark that any 3-regular, 3-edge connected graph is always simple, so the two entries in the bottom row correspond to the same family of graphs.

\begin{table}[h]
    \centering
    \begin{tabular}{|c|c|c|}
    \hline
         & All graphs & Simple graphs \\
         \hline
        Connected & $ 9\cdot 2^{n-3}$  \hfill (Corollary~\ref{coro:conn}) & $16 \cdot (\sqrt{3})^{n-5}$ \hfill (Conjecture~\ref{conj:simpconn3reg})\\
        \hline
        2-edge connected & $ 2^n+1$ \hfill (Corollary~\ref{coro:2conn3reg})   & $(\sqrt{3})^n+1$  \hfill (Conjecture~\ref{conj:simple2conn})\\
        \hline
        3-edge connected& $c \cdot 1.6779^n$ \hfill (Theorem~\ref{th:opt}) & $F_{n+2}+1$  \hfill (Conjecture~\ref{conj:3conn3reg})\\
        \hline
    \end{tabular}
    \caption{Upper bounds and conjectures for the maximal number of source-to-sink paths on acyclic $3$-regular graphs with $2n$ vertices.}
    \label{tab:my_label}
\end{table}

The extremal examples for this problem have the potential to be very interesting. This is, for instance, indicated by the (conjectured) Fibonacci example. Perhaps a deeper indication is as follows. These types of combinatorial problems \textit{should} favor the existence of a Hamiltonian path with monotone paths being able to take small jumps forward.  The absence of a Hamiltonian path forces any path to `make a choice' and this should lead to less paths overall.  Once there is a Hamiltonian path, the additional edges make it possible to jump forward -- for extremal graphs, these jumps should `mostly' be very short (long jumps get you closer to the sink). This combination of factors suggests that the extremizers for the general $\ell$-edge connected problem might have an interesting local structure.

\subsection{Motivation}
The main motivation behind this paper is the connection between monotone paths on 3-polytopes and the simplex algorithm, described in the works of~\cite{Loera2022,deLoera2023,Black2023}. The simplex algorithm has been an essential tool in Linear Programming since its discovery by George Dantzig~~\cite{Dantzig} in 1947. Despite the lack of sub-exponential bounds on its runtime and the development of two polynomial time algorithms for linear optimization, the interior-point method and the ellipsoid algorithm, the simplex algorithm is still widely used today as it requires on average only a linear number of steps of the algorithm. At each step, the algorithm updates the current best vertex using a \emph{pivot rule}, choosing in which direction the algorithm should move. 
Finding a pivot rule that guarantees only a sub-exponential number of steps, or constructing an exponential-time counter-example for a given rule, has received \emph{very} widespread attention since Klee and Minty's construction in the early 70's~\cite{KM}, see ~\cite{Amenta,Avis,avis2009postscript,friedmann2011subexponential,hansen2015improved,terlaky1993pivot,ziegler2004typical, Black2023}. The number of different paths the simplex algorithm can take when solving a linear program is linked to the number of $f$-monotone paths on 3-polytopes~\cite{Loera2022}. Counting the number of monotone paths is therefore an alternative to the problem of counting the different paths the simplex algorithm can take.

\subsection{Outline of the Argument}

In Section~\ref{sec:Hamil}, we begin by showing that source-to-sink paths are maximized on graphs with a directed Hamiltonian path. We show this constructively, by turning an acyclic 3-regular graph into one on a directed Hamiltonian path in a way that weakly increases the number of monotone paths. Then, in Section~\ref{sec:th3}, we optimize among graphs with a Hamiltonian path via a bijection to $n$-tuples. This optimization will complete the search for connected and 2-edge connected graphs with the maximum number of source-to-sink paths. To bound the family of $3$-regular, $3$-edge connected graphs on a Hamiltonian path, we enlarge our search to a slightly larger family of graphs. Namely, we study directed acyclic graphs on a directed Hamiltonian path such that: 

\begin{itemize}
    \item[a)] For the source $s$ and sink $t$, $\deg(s)=\deg(t)=2$. For any other vertex $v$, $\deg(v)=3$.
    \item[b)] There does not exist a proper interval on the Hamiltonian path $I=[y_i,y_j]$, $i<j$, such that at most two edges connect $I$ to the rest of the vertices. Note that there are always at least two edges for $i>1$ and $j<n$ as $(y_{i-1},y_i)$ and  $(y_{j},y_{j+1})$ are two of the edges of the Hamiltonian path. 
\end{itemize}
 Condition (a), allows for an upper bound on the number of source-to-sink paths up to a (small) constant factor.
Meanwhile, as shown in Proposition~\ref{3conn}, Condition (b), is encoding the $3$-edge connectivity of our graph. 
Finally, in Section~\ref{sec:opti}, we divide any optimal graph candidate into blocks of consecutive vertices on the Hamiltonian path, each of size between 35 and 40, where we can find the extremal structure with a finite amount of computation. This analysis yields the upper bound of $c \cdot 1.6779^n$ for 3-regular, 3-edge connected graphs. We remark that if Condition (b) is removed, then our bound of $c \cdot 1.6779^n$ no longer holds. Namely, graphs on $2n$ vertices of the form 
\[\begin{tikzcd}[column sep=0.18in]
	\bullet & \bullet & \bullet & \bullet & \bullet & \bullet & \bullet & \bullet & \cdots & \bullet & \bullet & \bullet & \bullet
	\arrow[from=1-1, to=1-2]
	\arrow[curve={height=-6pt}, from=1-1, to=1-3]
	\arrow[from=1-2, to=1-3]
	\arrow[curve={height=-6pt}, from=1-2, to=1-4]
	\arrow[from=1-3, to=1-4]
	\arrow[from=1-4, to=1-5]
	\arrow[from=1-5, to=1-6]
	\arrow[curve={height=-6pt}, from=1-5, to=1-7]
	\arrow[from=1-6, to=1-7]
	\arrow[curve={height=-6pt}, from=1-6, to=1-8]
	\arrow[from=1-7, to=1-8]
	\arrow[from=1-8, to=1-9]
	\arrow[from=1-9, to=1-10]
	\arrow[from=1-10, to=1-11]
	\arrow[curve={height=-6pt}, from=1-10, to=1-12]
	\arrow[from=1-11, to=1-12]
	\arrow[curve={height=-6pt}, from=1-11, to=1-13]
	\arrow[from=1-12, to=1-13]
\end{tikzcd}
\]
will have approximately $(\sqrt{3})^n\approx 1.733^n$ source-to-sink paths. 

\section{Proof of Theorem~\ref{th:Ham}}\label{sec:Hamil}

\subsection{Notations and conventions}
All directed graphs in this paper will contain no directed cycles and have a unique source and a unique sink. One may always equip an acyclic directed graph with a total order on its vertices. This total order $\preceq$ is such that the orientation of the edges in $D$ respects $\preceq$. That is, if the vertices $y_i \preceq y_j$ and $\{y_i,y_j\}$ is an edge, then this edge must be directed from $y_i$ to $y_j$. By an \textit{initial segment} of a total order on the set of vertices, we mean a set of (consecutive) vertices of the form $y_1,y_2,\ldots,y_k$ for some integer $k$. We will use $s$ and $t$ to denote our graphs' source and sink vertices respectively. A vertex has \textit{outdegree} (resp.~\textit{indegree}) $i$ if it is at the start (resp.~end) of exactly $i$ directed edges. For 3-regular graphs, we say a vertex is \textit{incoming} and is written $1$, if it has indegree at least two and \textit{outgoing} if it has outdegree at least two, written $0$. Motivated by the notation of Athanasiadis--De Loera--Zhang~\cite{Loera2022}, we use $\mu_D(y)$ to denote the number of directed paths from $s$ to $y$ for a vertex $y$, and $\mu_D(e)$ for the number of directed paths from $s$ to an initial vertex of an edge $e$ of $D$.

\subsection{Proof}

We begin with a structural result to characterize 3-edge connectedness for 3-regular graphs containing a Hamiltonian path.

\begin{proposition}\label{3conn}
    Let $D$ be a (not necessarily simple) 3-regular acyclic directed graph on a directed Hamiltonian path. Then $D$ is not 3 edge-connected if and only if $D$ contains at least one of the following structures. 
    \begin{itemize}
        \item In some initial segment of the ordering, there exist strictly more indegree two vertices than outdegree two vertices. 
        \item There exists some interval $[y_i,y_j]$ on our Hamiltonian path such that the only edges with exactly one endpoint in this interval are the two edges $(y_{i-1},y_{i})$ and $(y_{j},y_{j+1})$. 
    \end{itemize}
\end{proposition}
\begin{proof}
     Let $D$ be a 3-regular directed graph on a directed Hamiltonian path. Assume that $D$ is not 3-edge connected. Then, there exist two edges $e_1$ and $e_2$ for which their deletion causes $D$ to be broken into two disjoint nonempty subgraphs $A$ and $B$. Without loss of generality, we assume that the source vertex $s$ lies in $A$.
     \begin{itemize}
         \item[Case 1:] If both $e_1$ and $e_2$ are oriented towards $A$, then we have a contradiction because the Hamiltonian path cannot reach any vertex of $B$.
         \item[Case 2:] If both $e_1$ and $e_2$ are oriented from $A$ to $B$, then when following our Hamiltonian path, once we arrive at $B$, we cannot return to $A$. Therefore, $A$ must be an initial segment of our ordering and $B$ must be a final segment of our ordering. For any vertex $y_t$ in our ordering, we shall count the number of edges that have an initial endpoint of $y_i$ for $i\leq t$ and a final endpoint of $y_j$ for $j>t$. We start with three edges for $s=y_1$. When moving from $t$ to $t+1$, this edge count increases by one if $y_{t+1}$ is a vertex of outdegree two, and it decreases by one if $y_{t+1}$ is a vertex of indegree two. Since we end up with a count of two such edges ($e_1,e_2$), at the final vertex of $A$, there must have been more indegree two vertices in the initial segment $A$. 
         \item[Case 3:]  Finally, if $e_1$ and $e_2$ are oriented in opposing directions, then both lie on the Hamiltonian path since the graph is acyclic. However, in this case, $B=[y_i,y_j]$ and $e_1=(y_{i-1},y_i)$ and $e_2=(y_j,y_{j+1})$ for some choice of $i,j$. This yields the second structure and concludes the proof of the forward direction.
     \end{itemize} 
     We now prove the reverse direction. For the first structure, the condition of the number of incoming vertices being larger than the number of outgoing vertices in some initial segment $I$ implies (by the same argument as in Case 2 in the proof above) that there can be a maximum of two edges from $I$ to the complement of $I$. In the second case, deleting $(y_{i-1},y_i)$ and $(y_j,y_{j+1})$ shows that $D$ is not 3-edge connected. 
     \end{proof}
We note that the reverse implication of Proposition~\ref{3conn} holds for all $3$-regular directed graphs. We follow this first result with a lemma that will be needed to properly count the number of paths. Recall that $\mu_D(y)$ denotes the number of monotone paths from the source to a vertex $y$.

\begin{lemma}\label{lem:treesort}
    Let $(D,\preceq)$ be a 3-regular (not necessarily simple) acyclic directed graph. Then, there exists an ordering $\preceq'$ that preserves the orientation of all of the edges of $D$ and has the property that $\mu_D$ is weakly increasing on all vertices with respect to $\preceq'$.
\end{lemma}

\begin{proof}
     Let $(D,\preceq)$ be a 3-regular directed graph. Start at any outdegree two vertex and follow (backwards) the single incoming edge to its initial vertex. If this is another outdegree two vertex, repeat. If not, then stop. In this way, we can think of each outgoing vertex as belonging to a tree rooted at $s$ or at an indegree two vertex. Because there is a unique path from the root of each tree to each of its outgoing vertices, we find that $\mu$ is constant on each tree. Define $r_i$ to be the root of the $i^{\mathrm{th}}$ such tree (with respect to $\preceq$). Label the outgoing vertices as $o_1,o_2,\ldots,o_{n}$ according to $\preceq$. 
     
     We shall construct $\preceq'$ by reordering $\preceq$. Without loss of generality, assume $o_j$ belongs to the $i^{\mathrm{th}}$ tree. Then $r_i\preceq o_j$. If $r_{i+1} \succ o_j$, then keep the total order the same and consider $o_{j+1}$. If $r_{i+1} \prec o_j$, then rearrange the ordering by moving $o_j$ to be immediately preceding $r_{i+1}$. Observe that this will not change the orientation of any edge because the incoming edge of $o_j$ comes from a vertex strictly before $r_{i+1}$. In our new total order, all $o_k$'s with $k\leq j$ with root $r_i$ will have the property that $o_k\in [r_i,r_{i+1})$.
     Therefore, once this process is finished, the total order witnesses these tree subgraphs of $D$ by the fact that all vertices $y \in [r_i,r_{i+1})$ will belong to the $i^\mathrm{th}$ tree, $T_i$.
     
     Yet, it still may be possible for $\mu_D(r_i)>\mu_D(r_{i+1})$ in this ordering. However, if this is the case, we know that there is no edge between $T_i$ and $T_{i+1}$. Therefore, we may swap the whole trees $T_i$ and $T_{i+1}$ in the total order without changing the orientation of any edge of $D$. Thus, at the end of this process, we end up with a total order $\preceq'$ for $D$ on which $\mu_D$ is weakly increasing. 
\end{proof}

We can now start the proof of Theorem~\ref{th:Ham}. This theorem can be reformulated more precisely based on the previous lemma.

\textbf{Reformulation of Theorem~\ref{th:Ham}:}
    Let $D$ be a 3-regular (not necessarily simple), acyclic directed graph on $2n$ vertices equipped with the ordering provided by Lemma~\ref{lem:treesort}. Then, there exists a 3-regular directed graph $D'$ with a directed Hamiltonian path for which $\mu_D(y_i)\leq \mu_{D'}(y_i)$ for all $i\in [2n]$. Furthermore, if $D$ is 3-edge connected (resp.~2-edge connected, or simple), then so is $D'$.
\begin{proof}[Proof of Theorem~\ref{th:Ham}]
       We implement two types of local moves that change $D$ into our desired $D'$. Consider the set of outdegree two vertices of $D$ that are not connected to their preceding vertex. Label such outgoing vertices as $b_1,b_2,\ldots,b_k$. In $k$ steps, starting at $b_1$, and ending $b_k$, we shall connect $b_i$ to its preceding vertex in the total order at step $i$. At step $i$, denote the preceding vertex of $b_i$, as $p_i$. Notice that $p_i$ is not the root of the tree to which $b_i$ belongs. Hence, by the construction of our ordering, $p_i$ must be outgoing. Thus it has two outgoing edges $(p_{i},u_{i,1})$ and $(p_{i},u_{i,2})$ with $b_i \prec u_{i,1} \preceq u_{i,2}$. We label the incoming edge of $b_i$ as $(\ell_i, b_i)$. Delete the edges $(p_{i},u_{i,1}),(\ell_i,b_i)$ and add the edges $(\ell_i, u_{i,1}),(p_i,b_i)$. We include a schematic below of this change. The red edges are removed and replaced with the blue edges at each step.  
\[\begin{tikzcd}
	\ell_i &&& {p_{i}} & {b_i} && {u_{i,1}} && {u_{i,2}} \\
	\\
	\ell_i &&& {p_i} & {b_i} && {u_{i,1}} && {u_{i,2}}
	\arrow[red, curve={height=18pt}, from=1-1, to=1-5]
	\arrow["\times"{description}, dashed, from=1-4, to=1-5]
	\arrow[red, curve={height=-18pt}, from=1-4, to=1-7]
	\arrow[curve={height=-30pt}, from=1-4, to=1-9]
	\arrow[blue, curve={height=-18pt}, from=3-1, to=3-7]
	\arrow[blue, from=3-4, to=3-5]
	\arrow[curve={height=-18pt}, from=3-4, to=3-9]
\end{tikzcd}\]
Because the entire interval $[\ell_i, b_i]$ lies in the same tree, and $\mu$ is constant on each tree, the value of $\mu$ on every vertex remains the same. Therefore, for each $i\in [k]$, we perform these outgoing vertex edge swaps starting at $b_1$ and finishing at $b_k$ on $D$ until every outdegree two vertex is connected to the vertex preceding it in the total order. Call the resulting graph $\Tilde{D}$. 

Consider the set of indegree two vertices of $\tilde{D}$ that are not connected to their preceding vertex. Label such incoming vertices as $v_1,v_2,\ldots,v_h$. Denote their preceding vertices as $q_1,q_2,\ldots,q_h$. In $h$ steps, starting at $v_1$, and ending $v_h$, we shall connect $v_i$ to its preceding vertex in the total order at step $i$. At step $i$, denote the two incoming edges of $v_i$ as $(\ell_{i,1},v_i)$ and $(\ell_{i,2},v_i)$ where $\ell_{i,1}\preceq \ell_{i,2} \prec q_i$. Each $q_i$ must have an outgoing edge $(q_i,u_i)$. Delete the edges $(\ell_{i,2}, v_i),(q_{i},u_i)$ and add the edges $(\ell_{i,2}, u_i),(q_{i},v_i)$. We include a schematic below of this change. The red edges are removed and replaced with the blue edges at each step.

\[
\begin{tikzcd}
	\ell_{i,1} & \cdots & \ell_{i,2} & \cdots & {q_i} & {v_i} && u_i \\
	\\
	\ell_{i,1} & \cdots & \ell_{i,2} & \cdots & {q_{i}} & {v_i} && u_i
	\arrow[from=1-1, to=1-2]
	\arrow[curve={height=-18pt}, from=1-1, to=1-6]
	\arrow[from=1-2, to=1-3]
	\arrow[from=1-3, to=1-4]
	\arrow[red, curve={height=-18pt}, from=1-3, to=1-6]
	\arrow[from=1-4, to=1-5]
	\arrow["\times"{marking, allow upside down}, dashed, from=1-5, to=1-6]
	\arrow[red, curve={height=18pt}, from=1-5, to=1-8]
	\arrow[from=3-1, to=3-2]
	\arrow[curve={height=-18pt}, from=3-1, to=3-6]
	\arrow[from=3-2, to=3-3]
	\arrow[from=3-3, to=3-4]
	\arrow[blue, curve={height=-18pt}, from=3-3, to=3-8]
	\arrow[from=3-4, to=3-5]
	\arrow[blue, from=3-5, to=3-6]
\end{tikzcd}\]
We perform these incoming vertex moves for each $i\in [h]$ starting at $v_1$ and finishing at $v_h$ on $\tilde{D}$ until every (non-source) vertex is connected to the vertex preceding it in the total order. Upon completing all of these incoming vertex moves, we will have built a 3-regular graph $D'$ which contains a directed Hamiltonian path.

\begin{claim}
   When changing $D$ to $D'$, the number of directed paths from the source to each vertex weakly increases.
\end{claim}

Suppose $y_1,y_2,\ldots,y_{2n}$ are the vertices of $\tilde{D}$. For each $i\in [2n]$, define the set $S_{i}\coloneq \{j\in [h]: y_i\in [v_{j},u_j)\}$. To show $\mu_{D'}(y_i)\geq \mu_{D}(y_i)$ for all $i$, we shall prove the inequality: 
\begin{equation}\label{eq:mu}
    \mu_{D'}(y_i)-\mu_{D}(y_i)\geq \sum_{j\in S_i}\left(\mu_{D}(q_j)-\mu_{D}(\ell_{j,2})\right).
\end{equation}

Observe that by Lemma~\ref{lem:treesort}, the right-hand side of equation~\eqref{eq:mu} is nonnegative as $\mu$ is weakly increasing.
 Furthermore, to show $\mu_{D'}(y_i)\geq \mu_D(y_i)$ for all $i\in [2n]$, it suffices to prove equation~\eqref{eq:mu} for all $i < 2n$ because the number of source-to-sink paths can be counted as $3+\sum_{i=2}^{n}\mu(o_i)$. Indeed, each indegree two vertex simply appends an extra edge to each path leading to it, meanwhile each outgoing vertex takes every path going into it and splits it into two paths. We proceed via induction on $i$. 
 
Observe that all vertices in $[s,v_1)$ will keep the same value of $\mu$. On the right-hand side, we find that for all vertices $y_i\in [s,v_1)$, we have $S_{i}=\emptyset$. Assume equation~\eqref{eq:mu} holds for all vertices strictly before the vertex $y_{i}$. For the inductive step, we shall break into multiple, albeit short and similar, cases depending on how the incoming vertex moves affect $y_i$. Before commencing, observe that the vertices $u_j$ must all be incoming because all outgoing vertices will stay connected to their preceding vertices. Therefore if $y_i$ is outgoing, $S_i=S_{i-1}$. 

\textbf{Case 1:} $S_i=S_{i-1}$. Then, $y_i$ is connected to $y_{i-1}$ in $\tilde{D}$ and its incoming edges do not change from $\tilde{D}$ to $D'$. If $y_i$ is an outgoing vertex, then we may simply apply induction on $y_{i-1}$, noticing that $\mu_{D'}(y_i)=\mu_{D'}(y_{i-1})$ and $\mu_{D}(y_i)=\mu_{D}(y_{i-1})$. Thus, assume $y_i$ is incoming, and denote the non-Hamiltonian incoming edge of $y_i$ in $D'$ as $(x,y_i)$. Notice that because $S_i=S_{i-1}$, we find that $x$ is also connected to $y_i$ in $\tilde{D}$. Using the induction hypothesis, we obtain the following: 
\begin{align*}
    \mu_{D'}(y_i)-\mu_{D}(y_i)&=\mu_{D'}(y_{i-1})+\mu_{D'}(x)-\mu_{D}(y_{i-1})-\mu_{D}(x)\\
    &=\mu_{D'}(y_{i-1})-\mu_{D}(y_{i-1})+\mu_{D'}(x)-\mu_{D}(x)\\
    &\geq \sum_{j\in S_{i-1}} \left(\mu_{D}(q_j)-\mu_{D}(\ell_{j,2})\right)\\
    &= \sum_{j\in S_{i}}\left(\mu_{D}(q_j)-\mu_{D}(\ell_{j,2})\right).
\end{align*}
\textbf{Case 2:} $S_i=S_{i-1}\cup \{m\}$ for some $m\in [h]$. Then, at a single move, one of $y_i$'s edges is shifted to be part of the (partial) Hamiltonian path. More precisely, $y_i=v_{m}$ and $y_i\neq u_j$ for all $j$. Thus a single incoming edge of $y_i$ is changed when moving from $\tilde{D}$ to $D'$. Namely, we replace $(\ell_{m,2},y_i)$ with $(q_{m},y_i)=(y_{i-1},y_i)$. Consider the following: 
\begin{align*}
    \mu_{D'}(y_i)-\mu_{D}(y_i)&=\mu_{D'}(y_{i-1})+\mu_{D'}(\ell_{m,1})-\mu_D(\ell_{m,2})-\mu_D(\ell_{m,1})\\
    &\geq \mu_{D'}(y_{i-1})-\mu_D(\ell_{m,2})\\
    &=\mu_{D'}(y_{i-1})-\mu_D(y_{i-1})+\mu_D(y_{i-1})-\mu_D(\ell_{m,2})\\
    &\geq \left(\sum_{j\in S_{i-1}}\mu_D(q_j)-\mu_D(\ell_{j,2})\right)+\mu_D(q_{m})-\mu_D(\ell_{m,2})\\
    &=\sum_{j\in S_{i}}\left(\mu_D(q_{j})-\mu_D(\ell_{j,2})\right).
\end{align*}
\textbf{Case 3:} $S_{i-1}\setminus\{m\}=S_i$ for some $m\in [h]$. Then, $y_i$ is connected to its preceding vertex and at a single step, one of $y_i$'s edges is shifted backwards. More precisely, $y_i=u_{m}$ and $y_i\neq v_{j}$ for all $j$. Consider the following:
\begin{align*}
    \mu_{D'}(y_i)-\mu_D(y_i)&=\mu_{D'}(y_{i-1 })+\mu_{D'}(\ell_{m,2})-\mu_D(q_{m})-\mu_D(y_{i-1})\\
    &=\mu_{D'}(y_{i-1})-\mu_D(y_{i-1})+\mu_{D'}(\ell_{m,2})-\mu_D(q_{m})\\
    &\geq \left(\sum_{j\in S_{i-1}}\mu_D(q_j)-\mu_{D}(\ell_{j,2}) \right)+ \mu_{D}(\ell_{m,2})-\mu_D(q_{m})\\
    &=\sum_{j\in S_{i}}\left(\mu_D(q_j)-\mu_{D}(\ell_{j,2})\right).
\end{align*}
\textbf{Case 4:} $S_{i-1}\cup \{m_3\} \setminus \{m_1,m_2\}=S_i$ for some $m_1,m_2,m_3\in [h]$. Thus, $y_i$ is simultaneously $u_{m_1},u_{m_2}$ and $v_{m_3}$. There are two moves shifting each of $y_i$'s two incoming edges back and then a move that shifts one of these edges forward into the (partial) Hamiltonian path. More precisely, going into our final move, $y_i=v_{m_3}$ will be connected to both $\ell_{m_1,2}$ and $\ell_{m_2,2}$. Define $\ell^+\coloneq\ell_{m_3,2}=\text{max}_{\preceq}\{\ell_{m_1,2},\ell_{m_2,2}\}$ and $\ell^-\coloneq\ell_{m_3,1}=\text{min}_{\preceq}\{\ell_{m_1,2},\ell_{m_2,2}\}$. Therefore, we have $\mu_{D'}(y_i)=\mu_{D'}(y_{i-1})+\mu_{D'}(\ell^-)$ and $\mu_D(y_i)=\mu_D(q_{m_1})+\mu_D(q_{m_2})$. We may assume without loss of generality that $q_{m_1}\preceq q_{m_2}$ in our total order. Consider the following: 
\begin{align*}
    \mu_{D'}(y_i)-\mu_D(y_i)&=\mu_{D'}(y_{i-1})+\mu_{D'}(\ell^-)-\mu_D(q_{m_1})-\mu_D(q_{m_2})\\
    &= \mu_{D'}(y_{i-1})-\mu_D(y_{i-1})+\mu_D(y_{i-1})-\mu_D(\ell^+)\\
    &\ \ \ \ \ \ \ +\mu_D(\ell^+)-\mu_D(q_{m_1})+\mu_{D'}(\ell^-)-\mu_D(q_{m_2})\\
    &\geq \mu_{D'}(y_{i-1})-\mu_D(y_{i-1})+\mu_D(q_{m_3})-\mu_D(\ell_{m_3,2})\\
    &\ \ \ \ \ \ \ +\mu_D(\ell_{m_1,2})-\mu_D(q_{m_1})+\mu_{D}(\ell_{m_2,2})-\mu_D(q_{m_2})\\
    &\geq \sum_{j\in S_i}\left(\mu_D(q_j)-\mu_D(\ell_{j,2})\right).
\end{align*}

\textbf{Case 5:} $S_{i-1}\cup \{m_2\}\setminus\{m_1\}=S_i$ for some $m_1,m_2\in [h]$. Then, there is a move that shifts one of $y_i$'s edges back and a second move that shifts the closer edge forward to the (partial) Hamiltonian path. The two subcases distinguish between which edge is closer. More precisely, in $\tilde{D}$, $y_i=v_{m_2}=u_{m_1}$ is connected to $q_{m_1}$ and some vertex $w$. We have two subcases depending on $w$'s position relative to $\ell_{m_1,2}$. 
\newline
Case 5(a): Assume $w\prec \ell_{m_1,2}$. Then, we have $\mu_{D'}(y_i)=\mu_{D'}(y_{i-1})+\mu_{D'}(w)$ and $\mu_{D}(y_{i})=\mu_D(w)+\mu_D(q_{m_1})$. Consider the following: 
\begin{align*}
    \mu_{D'}(y_i)-\mu_{D}(y_i)&=\mu_{D'}(y_{i-1})+\mu_{D'}(w)-\mu_D(w)-\mu_D(q_{m_1})\\
    &\geq\mu_{D'}(y_{i-1})-\mu_D(q_{m_1})\\
    &\geq\mu_{D'}(y_{i-1})-\mu_D(y_{i-1})+\mu_D(y_{i-1})-\mu_D(\ell_{m_1,2})\\
    &~\ \ \ \ \ \ \ +\mu_D(\ell_{m_1,2})-\mu_D(q_{m_1})\\
    &=\mu_{D'}(y_{i-1})-\mu_D(y_{i-1})+\mu_D(q_{m_2})-\mu_D(\ell_{m_2,2})\\
    &~\ \ \ \ \ \ \ +\mu_D(\ell_{m_1,2})-\mu_D(q_{m_1})\\
    &\geq \sum_{j\in S_i} \left(\mu_{D}(q_j)-\mu_D(\ell_{j,2})\right).
\end{align*}
Case 5(b): Assume $w\succ \ell_{m_1,2}$. In this case, we have that $w=\ell_{m_2,2}$ and we have $\mu_{D'}(y_i)=\mu_{D'}(y_{i-1})+\mu_{D'}(\ell_{m_1,2})$ and $\mu_D(y_{i})=\mu_D(w)+\mu_D(q_{m_1})$. Consider the following: 
\begin{align*}
    \mu_{D'}(y_i)-\mu_D(y_i)&=\mu_{D'}(y_{i-1})+\mu_{D'}(\ell_{m_1,2})-\mu_D(w)-\mu_D(q_{m_1})\\
    &=\mu_{D'}(y_{i-1})-\mu_D(w)+\mu_{D'}(\ell_{m_1,2})-\mu_{D}(q_{m_1})\\
    &=\mu_{D'}(y_{i-1})-\mu_D(y_{i-1})+\mu_{D}(y_{i-1})-\mu_D(\ell_{m_2,2})\\
    &~\ \ \ \ \ \ \ +\mu_{D'}(\ell_{m_1,2})-\mu_{D}(q_{m_1})\\
    &\geq\mu_{D'}(y_{i-1})-\mu_D(y_{i-1})+\mu_{D}(q_{m_2})-\mu_D(\ell_{m_2,2})\\
    &~\ \ \ \ \ \ \ +\mu_{D}(\ell_{m_1,2})-\mu_{D}(q_{m_1})\\
    &\geq \sum_{j\in S_i}\left(\mu_{D}(q_j)-\mu_{D}(\ell_j)\right).
\end{align*}
We assert that these are all of the possible cases because $S_i$ can gain at most one index and lose at most two indices from $S_{i-1}$. Additionally, it is impossible for $S_i$ to lose two indices without gaining one, as we would not have a Hamiltonian path otherwise. Therefore, we have proven that our local moves could have only increased the value of $\mu$ on each vertex and thus the total number of paths on our graph. Next, we prove the latter claims made in the statement of this theorem. 
\begin{claim}
    If $D$ is simple, then so is $D'$. 
\end{claim}
Assume $D$ is simple and, for the sake of contradiction, that $D'$ is not. Then, because $D'$ is 3-regular, it contains a duplicate of an edge on the Hamiltonian path or a duplicate of an edge between $s$ and $t$. Because we only add an edge between two consecutive vertices when no edge is already present, the former cannot occur. Also, since $D$ is simple, $u_{i,1}\prec u_{i,2}$, so neither of the edges $(\ell_i,u_{i,1}),(p_i,b_i)$ may be from $s$ to $t$ in our outgoing vertex moves. Furthermore, since $D$ is simple, $\ell_{i,1} \prec \ell_{i,2}$, so neither of the edges $(\ell_{i,2},u_i),(q_i,v_i)$ can be between $s$ and $t$ in our incoming vertex moves. Hence we reach a contradiction and conclude that $D'$ must be simple. 
\begin{claim}
    If $D$ is 2-edge connected, then so is $D'$. 
\end{claim}
Observe that a 3-regular graph on a Hamiltonian path is not 2-edge connected if and only if there exists an edge of the path that is a bridge. Therefore, $D'$ is not 2-edge connected if and only if there exists strictly more incoming vertices than outgoing vertices in some initial segment of the path. This is because all indegree two vertices must pair with outdegree two vertices and two more indegree two vertices are required to pair with the (outgoing) source vertex. However, observe that our local moves preserve whether a vertex is incoming or outgoing. Therefore, if such a structure is present in $D'$, then $D$ cannot be 2-edge connected. 
\begin{claim}
    If $D$ is 3-edge connected, then so is $D'$. 
\end{claim}
For the final assertion, assume that $D$ is 3-edge connected. Assume for the sake of contradiction that $D'$ is not. Then, we must have one of the two forbidden structures listed in Proposition~\ref{3conn}. However, we can immediately rule out the first structure because these local moves preserve whether a vertex is incoming or outgoing. Thus, we must have a nonempty proper interval of vertices $I$ such that the only two edges that connect $I$ to its complement are part of the Hamiltonian path. Consider an outgoing move as follows:
 \[\begin{tikzcd}
	\ell_i &&& {p_{i}} & {b_i} && {u_{i,1}} && {u_{i,2}} \\
	\\
	\ell_i &&& {p_i} & {b_i} && {u_{i,1}} && {u_{i,2}}
	\arrow[red, curve={height=18pt}, from=1-1, to=1-5]
	\arrow["\times"{description}, dashed, from=1-4, to=1-5]
	\arrow[red, curve={height=-18pt}, from=1-4, to=1-7]
	\arrow[curve={height=-30pt}, from=1-4, to=1-9]
	\arrow[blue, curve={height=-18pt}, from=3-1, to=3-7]
	\arrow[blue, from=3-4, to=3-5]
	\arrow[curve={height=-18pt}, from=3-4, to=3-9]
\end{tikzcd}\]
Let $S$ be the subset of $\{\ell_i,p_i,b_i,u_{i,1}\}$ that is contained in $I$. If $S\neq \{p_i,b_i\}$, then the number of edges connecting $I$ to its complement did not change. If $S=\{p_i,b_i\}$, then we know that $(p_i,u_{i,2}),$ an edge going into the leftmost vertex of $I$, and an edge leaving the rightmost vertex of $I$, will still be present. Therefore, since there were three edges bridging $I$ to its complement in $D$, there will still be three edges bridging $I$ to its complement in $\tilde{D}$ after performing any outgoing vertex move. Consider an incoming vertex move as follows:
\[
\begin{tikzcd}
	\ell_{i,1} & \cdots & \ell_{i,2} & \cdots & {q_i} & {v_i} && u_i \\
	\\
	\ell_{i,1} & \cdots & \ell_{i,2} & \cdots & {q_{i}} & {v_i} && u_i
	\arrow[from=1-1, to=1-2]
	\arrow[curve={height=-18pt}, from=1-1, to=1-6]
	\arrow[from=1-2, to=1-3]
	\arrow[from=1-3, to=1-4]
	\arrow[red, curve={height=-18pt}, from=1-3, to=1-6]
	\arrow[from=1-4, to=1-5]
	\arrow["\times"{marking, allow upside down}, dashed, from=1-5, to=1-6]
	\arrow[red, curve={height=18pt}, from=1-5, to=1-8]
	\arrow[from=3-1, to=3-2]
	\arrow[curve={height=-18pt}, from=3-1, to=3-6]
	\arrow[from=3-2, to=3-3]
	\arrow[from=3-3, to=3-4]
	\arrow[blue, curve={height=-18pt}, from=3-3, to=3-8]
	\arrow[from=3-4, to=3-5]
	\arrow[blue, from=3-5, to=3-6]
\end{tikzcd}\]
Let $T$ be the subset of $\{\ell_{i,2},q_i,v_i,u_i\}$ that is contained in $I$. If $T\neq \{q_i,v_i\}$, then the number of edges connecting $I$ to its complement did not change. If $T=\{q_i,v_i\}$, then we know that $(\ell_{i,1},v_i)$, an edge going into the leftmost vertex of $I$, and an edge leaving the rightmost vertex of $I$, will still be present. Therefore, since there were three edges bridging $I$ to its complement in $\tilde{D}$, there will still be three edges bridging $I$ to its complement in $D'$ after performing any incoming vertex move.
Therefore, such an interval $I$ cannot exist, which means that $D'$ is 3-edge connected as desired. 
\end{proof}

\section{Technical results and supplementary bounds}\label{sec:th3}

\subsection{Construction of tuples} We suppose from now on that our graph has a Hamiltonian path and transition to a slightly different class of graphs as outlined in the introduction, directed acyclic graphs on a directed Hamiltonian path such that:
\begin{itemize}
    \item[a)] For the source $s$ and sink $t$, $\deg(s)=\deg(t)=2$. For any other vertex $v$, $\deg(v)=3$.
    \item[b)] There does not exist a proper interval on the Hamiltonian path $I=[y_i,y_j]$, $i<j$, such that at most two edges connect $I$ to the rest of the vertices. Note that there are always at least two edges for $i>1$ and $j<n$ as $(y_{i-1},y_i)$ and $(y_{j},y_{j+1})$ are two edges of the Hamiltonian path.
\end{itemize}

Recall that one can count the number of source-to-sink paths as $3+\sum_{i=2}^{n}\mu(o_i)$. Thus, if two graphs have the same values of $\mu$ for each of their $o_i$'s, then they must have the same number of paths. 

Let $D$ be a graph satisfying the conditions outlined above. 
We shall define an \textit{arc} of $D$ to be an edge of $D$ that is \textit{not} part of the Hamiltonian path. We label the arcs of $D$ from $1,2,\ldots,n$, according to the order of their outgoing vertices. With this labeling in place, we define $\rho(i)$ to be the largest label used before the incoming vertex of the $i^{th}$ arc. With this definition in place, we return an $n$-tuple of the form $(\rho(1),\ldots,\rho(n))$. 
\begin{example}
Consider the following graph:
\[\begin{tikzcd}[column sep=scriptsize]
	{\bullet_1} & {\bullet_2} & \bullet & {\bullet_3} & {\bullet_4} & \bullet & \bullet & {\bullet_5} & \bullet & \bullet
	\arrow[from=1-1, to=1-2]
	\arrow[curve={height=-12pt}, from=1-1, to=1-3]
	\arrow[from=1-2, to=1-3]
	\arrow[curve={height=-12pt}, from=1-2, to=1-6]
	\arrow[from=1-3, to=1-4]
	\arrow[from=1-4, to=1-5]
	\arrow[curve={height=-18pt}, from=1-4, to=1-9]
	\arrow[from=1-5, to=1-6]
	\arrow[curve={height=-12pt}, from=1-5, to=1-7]
	\arrow[from=1-6, to=1-7]
	\arrow[from=1-7, to=1-8]
	\arrow[from=1-8, to=1-9]
	\arrow[curve={height=-12pt}, from=1-8, to=1-10]
	\arrow[from=1-9, to=1-10]
\end{tikzcd}\]
This graph yields the $n$-tuple $(2,4,5,4,5)$. Indeed, we label the arcs 1 through 5 and $\rho(i)$ is the largest label left of the right endpoint of the $i^{th}$ arc.
\end{example}

 First, we observe that any $n$-tuple with the property that $\rho(i)\geq i$ will yield a graph on $2n$ vertices with $\text{deg}(s)=\text{deg}(t)=2$ and $\text{deg}(v)=3$ for all $v\neq s,t$, up to reordering the incoming vertices with the same assignment of $\rho$. However, notice that reordering vertices with the same assignment of $\rho$ will not change the value of $\mu$ for any $o_i$. Therefore, we may choose the order by which such arcs fall. For the set of incoming vertices sharing the same assignment of $\rho$, we shall always choose the ordering according to the increasing order of the associated outgoing vertices. Now, the interval connected condition is equivalent to saying that there does not exist $i,k$ with $1\leq i\leq k \leq n$ and $|i-k|<n-1$  such that $\{j\in [1,n] : \rho(j)\in [i,k]\}=[i,\ldots,k]$. Therefore, given an $n$-tuple satisfying $\rho(i)\geq i$ and this interval requirement, we acquire a graph in our desired class.
\begin{example}
    Consider the string $(4,4,4,4)$. This will yield the following graph. 
    \[\begin{tikzcd}
	{\bullet_1} & {\bullet_2} & {\bullet_3} & {\bullet_4} & \bullet & \bullet & \bullet & \bullet
	\arrow[from=1-1, to=1-2]
	\arrow[curve={height=-30pt}, from=1-1, to=1-5]
	\arrow[from=1-2, to=1-3]
	\arrow[curve={height=-30pt}, from=1-2, to=1-6]
	\arrow[from=1-3, to=1-4]
	\arrow[curve={height=-30pt}, from=1-3, to=1-7]
	\arrow[from=1-4, to=1-5]
	\arrow[curve={height=-30pt}, from=1-4, to=1-8]
	\arrow[from=1-5, to=1-6]
	\arrow[from=1-6, to=1-7]
	\arrow[from=1-7, to=1-8]
\end{tikzcd}\]
\end{example}
\begin{remark}
    Any $n$-tuple of the form $(n,n,n,\ldots,n)$ will yield a graph which has the property that $\mu_D(o_i)=1$ for all $i$. Therefore, these graphs provide the minimum source-to-sink path count of $n+1$.
\end{remark}

\subsection{Optimizing tuples}
The next lemma will imply that there cannot be many consecutive incoming vertices. In particular, we show that if two arcs have the same image $i$ under $\rho$, then either the graph is suboptimal or one of these arcs was the $i^{\text{th}}$ one. Furthermore, the following proofs will all be of the same flavor: change the tuple in a way that increases the number of paths and show that the graph is still sufficiently connected after this change. 

\begin{lemma}\label{lem:doublelabel}
    If there exist $j_1<j_2\leq i-1$ such that $\rho(j_1)=\rho(j_2)=i$, then the corresponding graph is suboptimal. 
\end{lemma}

\begin{proof}
    Assume that there exist $j_1<j_2\leq i-1$ such that $\rho(j_1)=\rho(j_2)=i$. We claim that by changing $\rho(j_1)$ to $i-1$ we obtain a valid assignment with more paths. Indeed, this change can only increase $\mu$ on each vertex and strictly increases $\mu$ on $o_i$. Observe that we still satisfy $\rho(j)\geq j$ for all $j$. Also, the interval condition still holds because for $h\geq 0$ if $j_1\in [i-h-1,\ldots,i-1]$, then $j_2\in [i-h-1,\ldots,i-1]$ while $\rho(j_2)=i$. 
\end{proof}

 In the following statement and in Section~\ref{sec:opti}, we shall use ``0" to denote an outgoing vertex and ``1" to denote an incoming vertex. This result will limit the number of consecutive $0$'s and consecutive $1$'s on an extremizer with a Hamiltonian path.
 
\begin{lemma}\label{lem:01}
    Any 3-regular, 3-edge connected acyclic graph maximizing the number of source-to-sink paths on a directed Hamiltonian path, cannot contain three consecutive $0$'s or three consecutive $1$'s.
\end{lemma}

\begin{proof}
    Suppose $D$ contains three consecutive incoming vertices coming from arcs $j_1<j_2<j_3$. then it must be that $j_1,j_2\leq i-1$ while $\rho(j_1)=\rho(j_2)=i$, for some $i\geq j_3$. By Lemma~\ref{lem:doublelabel}, we are suboptimal.
    
    Consider the reverse of $D$, denoted as $D_R$. This is the same underlying graph as $D$, but we switch the orientation on all edges. There is a natural bijection between source-to-sink paths of $D$ and $D_R$ by `traveling the path backwards', i.e. $(e_1,e_2,\ldots,e_k)$ becomes $(e_k,e_{k-1},\ldots,e_2,e_1)$. Suppose that $D$ has three consecutive outgoing vertices. Then, in the reverse of $D_R$, these three consecutive outgoing vertices become three consecutive incoming vertices. Therefore, we may apply the prescribed change in Lemma~\ref{lem:doublelabel} to increase the number of paths. This increases the number of paths in the reverse of $D_R$. Thus, it increases the number of paths in $D$. We conclude that if $D$ contains three consecutive outgoing vertices, then $D$ is suboptimal, which finishes the proof.
\end{proof}

\begin{remark}
    We remark that Lemma~\ref{lem:doublelabel} and Theorem~\ref{th:opt} still hold for 3-regular, 3-edge connected graphs under a mild adaptation of the construction, with the following conditions for $\rho$: 
    \begin{enumerate}
        \item[a.]$\rho(i)\geq i$, 
        \item[b.]$|\{j: \rho(j)\leq k\}|- k \geq 2$ for all $k<n$, 
    \item[c.] There does not exist a proper interval $[i,k]=\{j: \rho(j)\in [i,k]\}$.
    \end{enumerate} The proofs of Lemmas~\ref{lem:doublelabel} \& ~\ref{lem:01} then follow in the exact same way.

\end{remark}

\subsection{Supplementary Bounds and Conjectures}\label{subsec:supp}
Before moving on to the proof of Theorem~\ref{th:opt}, we use Theorem~\ref{th:Ham} and the tuples described above to give tight bounds for the connected and 2-edge connected cases.

\begin{corollary}\label{coro:conn}
    Let $D$ be a (not necessarily simple) 3-regular acyclic directed graph on $2n$ vertices with a unique source and a unique sink. Then $D$ has at most $9 \cdot 2^{n-3}$ source-to-sink paths. Furthermore, this bound is tight for all $n\geq 3$.
\end{corollary}
\begin{proof}
    Let $D$ be a (not necessarily simple) 3-regular acyclic directed graph on $2n$ vertices with a unique source and a unique sink. Then, we may apply Lemma~\ref{lem:treesort} and Theorem~\ref{th:Ham} to suppose it has a Hamiltonian path. We adapt the tuple construction from this section to 3-regular graphs on a Hamiltonian path by creating an $(n+1)$-tuple corresponding to a graph on $2n+2$ vertices and merging the first two vertices and last two vertices to get $2n$ total vertices. Because $s$ and $t$ now have degree 3, we require that $\rho(1)\geq 2$ and $\rho(j)=n+1$ occurs at least twice, and $\rho(i)\geq i$ for all $i\in [n+1]$. These conditions permit that in the graph reconstruction, one can merge the first two outgoing vertices into a source $s$ and the final two incoming vertices into a sink $t$. To optimize, we analyze two situations. 
    \begin{itemize}
    \item[Case 1]: Suppose either $\rho(1)$ or $\rho(2)$ equals $n+1$. Because the first two vertices will be merged, we may assume $n+1=\rho(1)$. Then, the maximizer must come from the $(n+1)$-tuple $(n+1,2,3,4,\ldots,n,n+1)$. Otherwise, we could decrease an assignment of $\rho$ to increase paths. It is straightforward to verify that such a graph yields a source-to-sink path count of $2^n+1$. 
    \item[Case 2]: Suppose $\rho(1),\rho(2)<n+1$. Then, any optimal graph must have the property that $\rho(1)=2=\rho(2)$. Otherwise, we could decrease an assignment of $\rho(1)$ or $\rho(2)$ to increase paths. Thus, the reverse graph, $D_R$, also has the property that $\rho(1),\rho(2)< n+1$. Therefore, we may apply the same argument on the reverse graph to say that $\rho(1)=2=\rho(2)$ in $D_R$. In $D$, this means $\rho(n)=n+1=\rho(n+1)$, and we still have $\rho(1)=2=\rho(2)$. The only remaining constraint is that $\rho(i)\geq i$, and decreasing any assignment of $\rho$ increases source-to-sink paths. Hence, the maximizer must satisfy $\rho(i)=i$ for $2\leq i <n$. It is straightforward to verify that such a graph has $9 \cdot 2^{n-3}$ source-to-sink paths. Therefore, the result follows as desired.
    \end{itemize}
\end{proof}

\begin{example}
    The extremum of Corollary~\ref{coro:conn} on ten vertices with corresponding tuple $(2,2,3,4,6,6)$ can be viewed pictorially as follows: 
  \[\begin{tikzcd}[column sep=0.24in]
	{\bullet_{1,2}} & \bullet & \bullet & {\bullet_3} & \bullet & {\bullet_4} & \bullet & {\bullet_5} & {\bullet_6} & \bullet
	\arrow[from=1-1, to=1-2]
	\arrow[curve={height=-6pt}, from=1-1, to=1-2]
	\arrow[curve={height=-12pt}, from=1-1, to=1-3]
	\arrow[from=1-2, to=1-3]
	\arrow[from=1-3, to=1-4]
	\arrow[from=1-4, to=1-5]
	\arrow[curve={height=-6pt}, from=1-4, to=1-5]
	\arrow[from=1-5, to=1-6]
	\arrow[from=1-6, to=1-7]
	\arrow[curve={height=-6pt}, from=1-6, to=1-7]
	\arrow[from=1-7, to=1-8]
	\arrow[from=1-8, to=1-9]
	\arrow[curve={height=-12pt}, from=1-8, to=1-10]
	\arrow[from=1-9, to=1-10]
	\arrow[curve={height=-6pt}, from=1-9, to=1-10]
\end{tikzcd}\]
This graph has $3\cdot 2 \cdot 2 \cdot 3=36$ paths as desired.
\end{example}

\begin{corollary}\label{coro:2conn3reg}
    Let $D$ be a (not necessarily simple) 2-edge connected, 3-regular acyclic directed graph on $2n$ vertices with a unique source and a unique sink. Then $D$ has at most $2^{n}+1$ source-to-sink paths. Furthermore, this bound is tight for all $n\geq 1$.
\end{corollary}

\begin{proof}
    Let $D$ be a (not necessarily simple) 2-edge connected, 3-regular acyclic directed graph on $2n$ vertices with a unique source and a unique sink. As in the previous proof, we may apply Lemma~\ref{lem:treesort} and Theorem~\ref{th:Ham} and adapt the construction before Example 1. Now, we require: $\rho(i)\geq i$, $\rho(1)\geq 2$, $\rho(i)=n+1$ occurs at least twice, and there does not exist some integer $1\leq k \leq n$ such that $\rho(i)\leq k$ for all $i\leq k$. Indeed, the last condition is because $D$ is $2$-edge connected. We shall argue that if $\rho(1),\rho(2)<n+1$, then $D$ is suboptimal. Without loss of generality, assume $M=\rho(1)\geq \rho(2)$. Then, because $D$ is 2-edge connected on a Hamiltonian path, there must exist some $1<j\leq M$ such that $N=\rho(j)>M$. Swap the assignments of $\rho$ between $1$ and $j$. That is, replace $M$ with $N$ in $\rho(1)$ and make $\rho(j)=M$. Recall that the number of source-to-sink paths can be counted as $3+\sum_{i=2}^{n}\mu_D(o_i)$. Furthermore, by counting paths according to the final arc traveled, we acquire \begin{equation}\label{eq:labelcount}
    \mu_D(o_{i})=\mu_D(\alpha_{i+1})=1+\sum_{k:\rho(k)\leq i}\mu_D(\alpha_k)
    \end{equation}
    
for $i>1$ with $\alpha_k$ denoting the $k^\text{th}$ arc. 
By equation~\eqref{eq:labelcount}, such a swap between $M$ and $N$ can only increase the number of source-to-sink paths. Indeed, if a sum had just one of these two arcs as a summand, it used to have the one with the smaller value of $\mu$. After the $\rho$-swap, any such sum will now have the one with the larger value of $\mu$, associated with the new $\rho(j)=M$. 
Furthermore, this operation swaps the positions of two incoming vertices, therefore it preserves the vertex type of all vertices in the total order. Therefore, $D$ remains 2-edge connected. Additionally, because in an optimum $D$, we have $\mu_D(o_j)\geq \mu_D(o_2)>1=\mu_D(o_1)$, such a change will strictly increase the number of paths to $o_M$ rendering $D$ suboptimal. With this restriction that $n+1=\rho(1)$, the best possible $(n+1)$-tuple is $(n+1,2,3,\ldots,n,n+1)$. It is straightforward to verify that this graph is 2-edge connected and has $2^n+1$ source-to-sink paths.
\end{proof}

\begin{example}
    The extremum of Corollary~\ref{coro:2conn3reg} on eight vertices with corresponding tuple given by $(5,2,3,4,5)$ can be viewed pictorially as follows: 
  \[\begin{tikzcd}
	{\bullet_{1,2}} & \bullet & {\bullet_3} & \bullet & {\bullet_4} & \bullet & {\bullet_5} & \bullet
	\arrow[from=1-1, to=1-2]
	\arrow[curve={height=-6pt}, from=1-1, to=1-2]
	\arrow[shift right, curve={height=12pt}, from=1-1, to=1-8]
	\arrow[from=1-2, to=1-3]
	\arrow[from=1-3, to=1-4]
	\arrow[curve={height=-6pt}, from=1-3, to=1-4]
	\arrow[from=1-4, to=1-5]
	\arrow[from=1-5, to=1-6]
	\arrow[curve={height=-6pt}, from=1-5, to=1-6]
	\arrow[from=1-6, to=1-7]
	\arrow[from=1-7, to=1-8]
	\arrow[curve={height=-6pt}, from=1-7, to=1-8]
\end{tikzcd}\]
Indeed, this graph has $2^4+1=17$ paths as desired. 
\end{example}

For simple graphs, we expect that the best way to maximize the path count is to repeat blocks of four vertices that each consist of two overlapping acyclic triangles. Such blocks contribute a factor of three to the path count, which leads to the factor(s) of $\sqrt{3}$ in our conjectures. In the connected case, we simply make the additional edges connected to $s$ and $t$ as short as possible. Meanwhile, in the 2-edge connected case, we use an $s-t$ edge to ensure connectivity just like in Corollary~\ref{coro:2conn3reg}. These expectations lead to the following conjectures.

\begin{conjecture}\label{conj:simpconn3reg}
    Let $D$ be a simple, 3-regular acyclic directed graph on $2n$ vertices with a unique source and a unique sink. Then, $D$ has at most $16 \cdot (\sqrt{3})^{n-5}$ source-to-sink paths. Furthermore, this bound is tight for odd values of $n\geq 5$.
\end{conjecture}

\begin{example} The family of graphs that we expect to be optimal in Conjecture~\ref{conj:simpconn3reg} consists of tuples of the form $(3,3,3,5,5,\ldots,n-2,n-2,n+1,n+1,n+1)$.
The following graph with tuple $(3,3,3,6,6,6)$ is the conjectured optimum for Conjecture~\ref{conj:simpconn3reg} on ten vertices.
   \[\begin{tikzcd}[column sep=scriptsize]
	\bullet_{1,2} & \bullet_3 & \bullet & \bullet & \bullet & \bullet_4 & \bullet_5 & \bullet_6 & \bullet & \bullet
	\arrow[from=1-1, to=1-2]
	\arrow[curve={height=12pt}, from=1-1, to=1-3]
	\arrow[curve={height=-12pt}, from=1-1, to=1-4]
	\arrow[from=1-2, to=1-3]
	\arrow[curve={height=-12pt}, from=1-2, to=1-5]
	\arrow[from=1-3, to=1-4]
	\arrow[from=1-4, to=1-5]
	\arrow[from=1-5, to=1-6]
	\arrow[from=1-6, to=1-7]
	\arrow[curve={height=-12pt}, from=1-6, to=1-9]
	\arrow[from=1-7, to=1-8]
	\arrow[curve={height=-12pt}, from=1-7, to=1-10]
	\arrow[from=1-8, to=1-9]
	\arrow[curve={height=12pt}, from=1-8, to=1-10]
	\arrow[from=1-9, to=1-10]
\end{tikzcd}\]
Indeed, for this example, we acquire $3^2+2\cdot 3 + 1=16$ source-to-sink paths.
\end{example}

\begin{conjecture}\label{conj:simple2conn}
    Let $D$ be a simple, 2-connected, 3-regular acyclic directed graph on $2n$ vertices with a unique source and a unique sink. Then $D$ has at most $(\sqrt{3})^n+1$ source-to-sink paths. Furthermore, this bound is tight for even values of $n\geq 2$.
\end{conjecture}

\begin{example}\label{ex:8} The family of graphs that we expect to be optimal in Conjecture~\ref{conj:simple2conn} consists of tuples of the form $(n+1,3,3,5,5,\ldots,n-1,n-1,n+1,n+1)$.
    The following graph with tuple $(5,3,3,5,5)$ is our conjectured extremum on eight vertices. 
   \[\begin{tikzcd}
	\bullet_{1,2} & \bullet_3 & \bullet & \bullet & \bullet_4 & \bullet_5 & \bullet & \bullet
	\arrow[from=1-1, to=1-2]
	\arrow[curve={height=-12pt}, from=1-1, to=1-3]
	\arrow[curve={height=12pt}, from=1-1, to=1-8]
	\arrow[from=1-2, to=1-3]
	\arrow[curve={height=-12pt}, from=1-2, to=1-4]
	\arrow[from=1-3, to=1-4]
	\arrow[from=1-4, to=1-5]
	\arrow[from=1-5, to=1-6]
	\arrow[curve={height=-12pt}, from=1-5, to=1-7]
	\arrow[from=1-6, to=1-7]
	\arrow[curve={height=-12pt}, from=1-6, to=1-8]
	\arrow[from=1-7, to=1-8]
\end{tikzcd}
\]
Indeed, this example has $3^2 + 1=10$ source-to-sink paths as desired. 
\end{example}

\section{Proof of Theorem~\ref{th:opt}}~\label{sec:opti}

\subsection{An Integer Programming Model}
Following these results, the candidate graphs for maximal number of paths are few enough that we can solve the problem in reasonable time for graphs of less than 40 vertices, and join the solutions together to generalize for larger $n$. We introduce Model~\ref{mod:2.5conn} to maximize the number of source-to-sink paths for the specific class of graphs used in Lemma~\ref{lem:01}. In this model, the $x_i$ represent the number of paths from vertex $1$ to vertex $i$, and the $a_{i,j}$ represent whether there is an edge from $i$ to $j$, $a_{i,j}=1$, or not, $a_{i,j}=0$. In other words, we are maximizing the total number of source-to-sink paths, $x_n$, by optimizing the edge assignment, while respecting a set of constraints guaranteeing the validity of the result.

\begin{subequations}\label{mod:2.5conn}
\begin{align}
    \max\;\; & x_n && \nonumber\\ 
    \text{s.t.}\;\; & \displaystyle a_{i,i+1}=1 && \forall i,=0,\dots,n,\label{eq:1}\\
    & \displaystyle \sum_{j=0}^{i-1} a_{j,i}+\sum_{i+1}^{n+1} a_{i,j}=3&& \forall i=1,\dots,n
    \label{eq:2}\\
    & a_{j,i}=0 && \forall i,j=0,\dots,n+1,\, i\leq j \label{eq:3}\\
    & x_i=\sum_{j=0}^{i-1} x_ja_{j,i}  && \forall i=1,\dots,n,\label{eq:4}\\
    & \sum_{k=i}^{j}\left(\sum_{l=j+1}^{n+1}a_{k,l}+\sum_{l=0}^{i-1} a_{l,k}\right)\geq 3 && \forall i=1,\ldots,n-1\, ,j=2,\dots,n,\, i<j \label{eq:5}\\
    & x_0=x_1=1\label{eq:6}\\
    & x_i \in \mathbb{Z} && \forall i=0,\ldots,n\\
    & a_{i,j} \in \{0,1\} && \forall i,j=0,\ldots,n+1
\end{align}
\end{subequations}

As explained in more detail in the proof of Lemma~\ref{lemma:correct}, the different constraints correspond to specific properties our graph must have. Constraints~\eqref{eq:1} correspond to the existence of the Hamiltonian path (there is an edge from vertex $i$ to $i+1$ for all possible $i$). Constraints~\eqref{eq:2} ensure 3-regularity, each vertex has in total 3 incoming or outgoing edges. Finally, Constraints~\eqref{eq:5} guarantee that the graph is 3-edge connected, using the characterization from Proposition~\ref{3conn}. It also enforces that there is at least one edge outside the Hamiltonian path connected to vertex $0$ or vertex $n$. This last element guarantees that if we consider a block of consecutive vertices as the output of the model, then deleting the edges on the Hamiltonian path is not sufficient to disconnect the graph. 

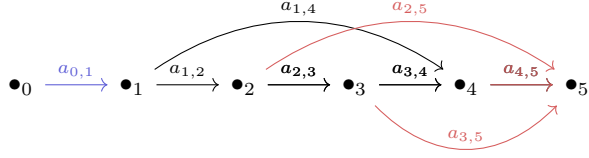
\begin{figure}[h]
  \[\begin{tikzcd}
	{\bullet_0} & {\bullet_1} & {\bullet_2} & {\bullet_3} & {\bullet_4} & {\bullet_5} 
	\arrow[color={rgb,255:red,92;green,92;blue,214}, from=1-1, to=1-2, "a_{0,1}"]
	\arrow[from=1-2, to=1-3,"a_{1,2}"]
    \arrow[from=1-3, to=1-4, "a_{2,3}"]
    \arrow[from=1-4, to=1-5, "a_{3,4}"]
    \arrow[from=1-5, to=1-6, "a_{4,5}"]
    \arrow[ curve={height=-30pt}, from=1-2, to=1-5, "a_{1,4}"]
    \arrow[from=1-3, to=1-4, "a_{2,3}"]
    \arrow[color={rgb,255:red,214;green,92;blue,92}, curve={height=-30pt}, from=1-3, to=1-6, "a_{2,5}"]
    \arrow[color={rgb,255:red,214;green,92;blue,92}, curve={height=30pt}, from=1-4, to=1-6, "a_{3,5}"]
    \arrow[from=1-4, to=1-5, "a_{3,4}"]
    \arrow[color={rgb,255:red,214;green,92;blue,92}, from=1-5, to=1-6, "a_{4,5}"] 
\end{tikzcd}\]
\caption{The $a_{i,j}$ displayed correspond to those equal to 1: we have chosen these arcs to be in the graph. Each $x_i$ corresponds to the number of paths leading to vertex $i$: here, $x_0=x_1=x_2=x_3=1$, $x_4=2$ and $x_5=3$. Vertices $0$ and $5$ respectively represent any vertices that come before vertex 1 or after vertex 4. On this small example, only arc $(1,4)$ is contained in the block of vertices and is not on the Hamiltonian path.}
\end{figure}

\subsection{Technical remarks}
 Vertices $0$ and $n+1$ are not real vertices but represent respectively any vertex that could come before vertex 1 or after vertex $n$. While these dummy vertices are not relevant for the first and last blocks of the construction, they are essential in the construction of the smaller inner blocks. Indeed, we will be trying to bound the worst increase in the number of paths that can happen on a block of consecutive vertices. These dummy vertices represent arcs coming from vertices before (dummy vertex 0), or going to vertices after (dummy vertex $n+1$), the block we optimize on. Note that setting the number of paths to vertex 0, $x_0$, equal to the number of paths to vertex 1, $x_1$, gives an upper-bound: this is the highest possible value for $x_0$ as the number of paths is monotonically increasing along the Hamiltonian path.

 For clarity, we write the model with a product of $a_{i,j}$, the binary variable indicating whether there is an edge from node $i$ to node $j$, and $x_j$, the (bounded) variable counting the number of paths from node $1$ to node $j$. This could be linearized using the bound on $x_j$ to keep this as an Integer Program. Finally, values were computed using Gurobi 10.0.1, implemented in Julia using the JuMP package.

\subsection{Theoretical guarantees}
\begin{lemma}\label{lemma:correct}
    Model~\ref{mod:2.5conn} computes the maximal number of paths from source to sink for a 3-regular, 3-edge connected graph with a Hamiltonian path and a dummy source vertex.
\end{lemma}

\begin{proof}
    We first verify that the number of paths is correctly counted by $x_i$ (there is no $x_{n+1}$). We have that $x_0=1$ and $x_1=1$, which are the correct counts as the only edge between these two vertices is the one associated with the Hamiltonian path. For any vertex after the first, the number of paths is given by Constraint~\eqref{eq:4}: the number of paths $x_i$ to get to $i$ corresponds to the sum of the number of paths to get to any previous vertex $j$, $x_j$, times the boolean variable $a_{j,i}$ representing the existence of the edge $(j,i)$. This is exactly the sum over all edges incoming to $i$ of the number of paths that led to the sources of the edges. By induction, the $x_i$ count the number of paths correctly. The variable $x_n$ therefore represents the number of paths from vertex 1 to vertex $n$ in a graph with a secondary dummy source vertex such that $x_0=x_1$.\\

    The other constraints ensure that the resulting graph has the desired structure. Constraints~\eqref{eq:1} guarantee the presence of a Hamiltonian path going from $1$ to $n$ passing through all the other vertices in order, as shown in Theorem~\ref{th:Ham}. It also includes the connection of vertices 1 and $n$ with the dummy vertices, to ensure vertices 1 and $n$ are not connected to too many vertices in the graph. Constraints~\eqref{eq:2} count the number of edges incident to a vertex $i$, which has to be exactly 3 since the graph is 3-regular. Constraints~\eqref{eq:3} guarantee that all edges are oriented from source to sink. The combination of these constraints ensures that any solution of the model is a directed acyclic 3-regular graph. Finally, Constraints~\eqref{eq:5} ensure that the graph cannot be disconnected by removing only 2 edges. For this, we consider any block of consecutive vertices $i$ to $j$. For any such block, we enforce that at least 3 edges have an endpoint either smaller than $i$ or greater than $j$. In other words, it is impossible to disconnect the block of vertices by simply removing the two edges from the path $(1,2,\ldots,n)$, as shown in Proposition~\ref{3conn}.
\end{proof}

We point out that the model does not compute the maximal scaling of the number of paths for a sequence of $n$ vertices, only an upper bound. The dummy variable associated with $x_0=1$ is a worst-case scenario as it considers that \emph{all} outgoing vertices $u$ preceding the first $0$ (vertex $v_1$, starting the $n$ vertex block), and whose other endpoint are in our block of vertices, have $\mu(u)=\mu(v_1)$. This is not always the case, as in general we only have $\mu(u)\leq \mu(v_1)$. We now bound the number of paths for any graph of size $2n$ by grouping these blocks together.

\begin{proof}[Proof of Theorem~\ref{th:opt}]

The main part of the argument is to justify that we can decompose the $2n$ vertices into blocks of varying sizes, where the possible sizes are in some range $\{k,\ldots,k+5\}$. For each block, we can upper-bound the growth-rate of $\mu$ between the first and last vertex using the optimization model. We can then combine the different bounds multiplicatively over the different blocks to guarantee the overall bound.

We begin with the block decomposition. Firstly, by Lemma~\ref{lem:01}, we are guaranteed to see a vertex with two outgoing edges (a ``0'') at most every 3 vertices, and the first vertex is necessarily a $0$. Every time we encounter such a vertex, Model~\ref{mod:2.5conn} guarantees that the number of paths over the next $k$ steps increases at most by a certain factor $f(k)$. We are however not guaranteed to have a vertex with two incoming edges (a ``1'') followed by a $0$ exactly after $k$ steps, which would allow us to start the next block from this new $0$ vertex. Again using Lemma~\ref{lem:01}, we know that we only need to compute the worst possible growth for a small range of possible values: $k$ to $k+5$. Indeed, given that we cannot have three $0$'s or three $1$'s in a row, a ``$10$'' sequence is guaranteed at least every $5$ steps. 
For any 3-regular, 3-edge connected acyclic graph, starting from the first vertex, we can find a block of length $k+i$, with $0\leq i \leq 5$ such that the elements in positions $k+i$ and $k+i+1$ are respectively 1 and 0. Removing these first $k+i$ vertices and repeating this procedure, we obtain the desired block decomposition, with a final block of length less than $k$.

We now wish to bound the total growth. On each block starting in $a$, we know that $\mu(a+k+i)$ is upper-bounded by $\mu(a)f(k+i)$. Furthermore, by construction, vertex $a+k+i$ is an incoming vertex, and $a+k+i+1$ an outgoing vertex, they must share the same number of paths from the source by the Hamiltonian path construction. We therefore have that $\mu(a+k+i+1)\leq \mu(a)f(k+i)$.

Using Lemma~\ref{lemma:correct}, we obtain the maximal number of paths for the different values of $k+i$, described in Table~\ref{tab:growth}. For each of these, we know that this increase required using $k+1$ edges on the Hamiltonian path (the $k$ of the block, plus the final edge for the next $0$). The growth rate per vertex therefore corresponds to the $1/k$-th power of the maximal number of paths. Multiplying these successive growth rates together gives a bound of $(\max_{k=35,\ldots,40}g(k))^{2n-d}$ for the number of paths to go from the source to the $(2n-d)^{\mathrm{th}}$ vertex, where $d<35$ (before the final block). This last block has length less than 35, it can only increase the total number of paths by some factor $c$ at most. For each of the previous blocks, as shown in Table~\ref{tab:growth}, the worst case corresponds to exclusively blocks of length $36$, giving us a bound of $c \cdot 1.6779^{n-\frac{d}{2}}$, upper-bounded by $c\cdot1.6779^n$. To extend the result from the graphs in question to 3-regular graphs, we remark that the number of paths can only increase by a constant factor.
\begin{table}[h]
    \centering
    \begin{tabular}{|c|c|c|c|c|c|c|}
    \hline
        Number of vertices $k$ & 35 &36 &37 &38& 39&40 \\
        \hline
        Maximal number of paths & $8\,233$& $11\,117$& $14\,033$& $17\,293$&$22\,781$& $28\,726$\\
        \hline
        Squared growth factor $g(k)^2$& 1.6740& 1.6779 & 1.6756& 1.6713&1.6729& 1.6707\\
        \hline
    \end{tabular}
    \caption{Growth rate of the number of paths depending on block size.}
    \label{tab:growth}
\end{table}
\end{proof}

We wish to point out two elements from this proof. First, while the constant $c$ may seem very large as it corresponds to the maximal number of paths for a block of length less than 35, it is largely a trade-off with the crude upper-bound where we ignore $1.6779^{-\frac{d}{2}}$ in the final step.  Secondly, from the table above and Figure~\ref{fig:growth}, it is clear that better results could be obtained by considering larger blocks. This was out of reach of our experiments on a single laptop, but the model remains correct for larger values of $n$ and the bound in Theorem~\ref{th:opt} could naturally be improved. We note that for $k=21$, the growth per vertex is $1.7108$ and, while not monotonically decreasing, the trend between $k=21$ and $k=40$ is clearly indicating a smaller constant when larger blocks are considered.
\begin{figure}[h]
    \centering\includegraphics[width=0.5\linewidth]{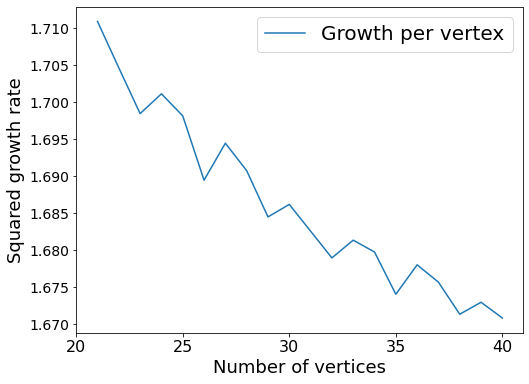}
    \caption{Evolution of the growth factor per vertex based on the choice of $k$.}
    \label{fig:growth}
\end{figure}

\section{Future Work}\label{sec:future}

This work highlights that for $3$-regular graphs obeying $\ell$-edge connectivity, there is a naturally associated family of directed graphs that maximize the source-to-sink path count. Therefore, this begs the question of finding bounds for general $k$-regular and $\ell$-edge connected graphs. Furthermore, for $k=\ell=3$, as noted in Figure 7 of Athanasiadis-De Loera-Zhang~\cite{Loera2022}, we see that the wedge of an ($n+1$)-gon over  an edge meets the conjectured bound. In general, does there exist a simple $d$-polytope whose graph maximizes the source-to-sink path count among all $d$-regular, $d$-edge connected acyclic directed graphs with a unique source and a unique sink? 

\begin{question}
    What is the maximum number of source-to-sink paths over all $k$-regular, $\ell$-edge connected acyclic directed graphs on $n$ vertices? What is this count over all such simple graphs? In both cases, can we characterize the extremal examples? 
\end{question}

\section{Acknowledgements}
The authors are very grateful to Stefan Steinerberger for the fruitful discussions and general advice. DG thanks Joshua Hinman for introducing him to the problem and for illustrative preliminary discussions, Isabella Novik for sage advice, helpful suggestions and expository guidance, and Germain Poullot for encouragement and helping build intuition for the problem.
\bibliographystyle{alpha}
\bibliography{biblio}

\begin{thebibliography}{BDLLS23}

\bibitem[ADLZ22]{Loera2022}
C.~A. Athanasiadis, J.~A. De~Loera, and Z.~Zhang.
\newblock Enumerative problems for arborescences and monotone paths on polytope graphs.
\newblock {\em J. Graph Theory}, 99(1):58--81, 2022.

\bibitem[AF17]{Avis}
D.~Avis and O.~Friedmann.
\newblock An exponential lower bound for cunningham’s rule.
\newblock {\em Mathematical Programming}, 161:271--305, 2017.

\bibitem[Avi09]{avis2009postscript}
D.~Avis.
\newblock Postscript to “what is the worst case behavior of the simplex algorithm?”.
\newblock {\em Polyhedral Computation}, 48:145--148, 2009.

\bibitem[AZ99]{Amenta}
N.~Amenta and G.~Ziegler.
\newblock Deformed products and maximal shadows of polytopes.
\newblock {\em Contemporary Mathematics}, pages 57--90, 1999.

\bibitem[BDL23]{Black2023}
A.~E. Black and J.~A. De~Loera.
\newblock Monotone paths on cross-polytopes.
\newblock {\em Discrete Comput. Geom.}, 70(4):1245–1265, October 2023.

\bibitem[BDLLS23]{deLoera2023}
A.~E. Black, J.~A. De~Loera, N.~L\"{u}tjeharms, and R.~Sanyal.
\newblock The polyhedral geometry of pivot rules and monotone paths.
\newblock {\em SIAM Journal on Applied Algebra and Geometry}, 7(3):623--650, 2023.

\bibitem[DMT16]{Dumitrescu2016}
A.~Dumitrescu, R.~Mandal, and C.~D. T{\'o}th.
\newblock Monotone paths in geometric triangulations.
\newblock In Veli M{\"a}kinen, Simon~J. Puglisi, and Leena Salmela, editors, {\em Combinatorial Algorithms}, pages 411--422, Cham, 2016. Springer International Publishing.

\bibitem[FHZ11]{friedmann2011subexponential}
O.~Friedmann, T.~D. Hansen, and U.~Zwick.
\newblock Subexponential lower bounds for randomized pivoting rules for the simplex algorithm.
\newblock In {\em Proceedings of the forty-third annual ACM symposium on Theory of computing}, pages 283--292, 2011.

\bibitem[GBD55]{Dantzig}
P.~Wolfe G.~B.~Dantzig, A.~Orden.
\newblock The generalized simplex method for minimizing a linear form under linear inequality restraints.
\newblock {\em Pacific Journal of Mathematics}, 5(2):183--195, 1955.

\bibitem[HZ15]{hansen2015improved}
T.~D. Hansen and U.~Zwick.
\newblock An improved version of the random-facet pivoting rule for the simplex algorithm.
\newblock In {\em Proceedings of the forty-seventh annual ACM symposium on Theory of computing}, pages 209--218, 2015.

\bibitem[JP25]{juhnke2025}
M.~Juhnke and G.~Poullot.
\newblock Unimodality of the number of paths per length on polytopes: Examples, counter-examples, and central limit theorem.
\newblock arXiv:2504.20739v2, 2025.

\bibitem[KM72]{KM}
V.~Klee and G.~J. Minty.
\newblock How good is the simplex algorithm?
\newblock {\em Inequalities}, 3(3):159--175, 1972.

\bibitem[LG24]{Sudakov}
B.~Sudakov L.~Gishboliner, Z.~Jin.
\newblock Ramsey problems for monotone paths in graphs and hypergraphs.
\newblock {\em Combinatorica}, 44:509--529, 2024.

\bibitem[McD95]{MCDONALD1995213}
J.~McDonald.
\newblock Fiber polytopes and fractional power series.
\newblock {\em Journal of Pure and Applied Algebra}, 104(2):213--233, 1995.

\bibitem[Pou24]{poullot2024}
G.~Poullot.
\newblock Vertices of the monotone path polytopes of hypersimplicies.
\newblock arXiv:2411.14102v1, 2024.

\bibitem[Ste28]{Steinitz}
E.~Steinitz.
\newblock Über isoperimetrische probleme bei konvexen polyedern.
\newblock {\em Journal für die reine und angewandte Mathematik}, 1928(159):133--143, 1928.

\bibitem[TZ93]{terlaky1993pivot}
T.~Terlaky and S.~Zhang.
\newblock Pivot rules for linear programming: a survey on recent theoretical developments.
\newblock {\em Annals of operations research}, 46:203--233, 1993.

\bibitem[Zie94]{Ziegler}
G.~M. Ziegler.
\newblock {\em Lectures on Polytopes}.
\newblock Springer New York, NY, 1994.

\bibitem[Zie04]{ziegler2004typical}
G.~M. Ziegler.
\newblock Typical and extremal linear programs.
\newblock In {\em The Sharpest Cut: The Impact of Manfred Padberg and His Work}, pages 217--230. SIAM, 2004.

\end{thebibliography}
\end{document}